\documentclass[oneside,a4paper,11pt,notitlepage]{article}
\usepackage[left=0.9in, right=0.9in, top=1.5in, bottom=1.5in]{geometry}
\usepackage{graphicx} 

\usepackage[T1]{fontenc} 
\usepackage[utf8]{inputenc} 
\usepackage[english]{babel} 
\usepackage{lipsum} 
\usepackage{lmodern}
\usepackage{amssymb}
\usepackage{amsthm}
\usepackage{bm}
\usepackage{mathtools}
\usepackage{xcolor}
\usepackage{dsfont}
\usepackage{enumerate}

\usepackage{yhmath}
\usepackage{stmaryrd}
\usepackage{amsmath}
\usepackage{tikz}
\usepackage{verbatim}
\usepackage{graphicx}
\usepackage{enumitem}
\usepackage{footnote} %
\newcommand{\SSS}{\color{red}}

\newcommand{\BBB}{\color{black}}

\usepackage{braket}
\usepackage{esint}
\newcommand{\abs}[1]{{\left|#1\right|}}
\newcommand{\norma}[1]{{\left\Vert#1\right\Vert}}

\usepackage{booktabs}
\usepackage{graphicx}
\usepackage{tikz}

\usepackage{multicol}
\usepackage{caption}
\usepackage{enumerate}
\usepackage[skins,theorems]{tcolorbox}
\tcbset{highlight math style={enhanced,
		colframe=black,colback=white,arc=0pt,boxrule=1pt}}
\theoremstyle{definition}
\newtheorem{definizione}{Definition}[section]
\theoremstyle{plain}
\newtheorem{teorema}{Theorem}[section]

\newtheorem{lemma}[teorema]{Lemma}
\newtheorem{prop}[teorema]{Proposition}
\newtheorem{corollario}[teorema]{Corollary}
\theoremstyle{definition}
\newtheorem{esempio}{Example}[section]
\newtheorem{oss}[esempio]{Remark}

\DeclareMathOperator{\R}{\mathbb{R}}

\DeclareMathOperator{\Om}{\Omega}

\usepackage{tikz}

\usetikzlibrary{patterns, calc, math, intersections}

\pgfdeclarepatternformonly{my dots}{\pgfqpoint{-1pt}{-1pt}}{\pgfqpoint{1.5pt}{1.5pt}}{\pgfqpoint{2pt}{2pt}}
{
    \pgfpathcircle{\pgfqpoint{0.5pt}{0.5pt}}{.3pt}
    \pgfpathcircle{\pgfqpoint{1.5pt}{1.5pt}}{.3pt}
    \pgfusepath{fill}
}

\def\Xint#1{\mathchoice
    {\XXint\displaystyle\textstyle{#1}}%
    {\XXint\textstyle\scriptstyle{#1}}
    {\XXint\scriptstyle\scriptscriptstyle{#1}}%
    {\XXint\scriptscriptstyle\scriptscriptstyle{#1}}%
      \!\int}
\def\XXint#1#2#3{{\setbox0=\hbox{$#1{#2#3}{\int}$}
    \vcenter{\hbox{$#2#3$}}\kern-.5\wd0}}
\def\dashint{\Xint-}
\usepackage{hyperref}
\hypersetup{linktoc=none, bookmarksnumbered, colorlinks=true, linkcolor=magenta, citecolor=cyan}

\title{A quantitative Talenti-type comparison result with Robin boundary conditions  }
\author{Vincenzo Amato, Rosa Barbato, Simone Cito, Alba Lia Masiello, Gloria Paoli}
\date{\today}

\newcommand{\Addresses}{{
\bigskip 
  \footnotesize 
  \textit{E-mail address}, V. ~Amato: \texttt{v.amato@ssmeridionale.it} 
 \medskip 
 
  \noindent\textsc{Mathematical and Physical Sciences for Advanced Materials and Technologies, Scuola Superiore Meridionale, Largo San Marcellino 10, 80138 Napoli, Italy. } 
 
  \footnotesize 
   \medskip

  \textit{E-mail address}, R.~Barbato: \texttt{rosa.barbato2@unina.it} 
  
   \medskip

     \textit{E-mail address}, G.~Paoli: \texttt{gloria.paoli@unina.it} 
   \medskip 
   
 \textsc{Dipartimento di Matematica e Applicazioni ``R. Caccioppoli'', Universit\`a degli studi di Napoli Federico II, Via Cintia, Complesso Universitario Monte S. Angelo, 80126 Napoli, Italy.}
    \medskip

   \footnotesize
 \textit{E-mail address}, S.~Cito: simone.cito@unisalento.it\texttt{@.it} 
  
   \medskip 
   
 \textsc{Dipartimento di Matematica e Fisica “E. De Giorgi“, Università del Salento, Via per Arnesano,
73100 Lecce, Italy.
 	}

   \medskip

   \footnotesize 
  \textit{E-mail address}, A.L.~Masiello: \texttt{masiello@altamatematica.it} 
  
   \medskip 
   
 \textsc{
 	Holder of a research grant from Istituto Nazionale di Alta Matematica "Francesco Severi" at Dipartimento di Matematica e Applicazioni "R. Caccioppoli", Via Cintia, Complesso Universitario Monte S. Angelo, 80126 Napoli, Italy.}

 \par\nopagebreak 

}}

\begin{document}

\maketitle
\begin{abstract}
The purpose of this paper is to establish a quantitative version of the Talenti comparison principle for solutions to the Poisson equation with Robin boundary conditions. This quantitative enhancement is proved in terms of the asymmetry of domain. The key role is played by a careful analysis of the propagation of asymmetry for the level sets of the solutions of a PDE. As a byproduct, we obtain an alternative proof of the quantitative Saint–Venant inequality for the Robin torsion and, in the planar case, of the quantitative Faber–Krahn inequality for the first Robin eigenvalue. In addition, we complete the framework of the rigidity result of the Talenti inequalities with Robin boundary conditions.


\end{abstract}
\section{Introduction}

Let $\Omega$ be an open, bounded, Lipschitz  set in $\R^n$, let $f\in L^2(\Omega)$ be a non-identically zero and nonnegative function, and let $\beta$ be a positive parameter. We consider the following problem
\begin{equation}\label{main_problem}
    \begin{cases}
    -\Delta u=f \, &\text{in } \Om, \\[1ex]
    \displaystyle{\frac{\partial u}{\partial\nu}+\beta u=0} & \text{on } \partial\Omega,
    \end{cases}
\end{equation}    
 and its symmetrized version 
\begin{equation}
\label{sym_problem}
      \begin{cases}
    -\Delta v=f^\sharp \, & \text{in } \Om^\sharp, \\[1ex]
    \displaystyle{\frac{\partial v}{\partial\nu}+\beta v=0} & \text{on } \partial\Omega^\sharp,
    \end{cases}
\end{equation}
where $\Omega^{\sharp}$ is the ball centered at the origin with the same measure as $\Omega$ and $f^\sharp$ is the Schwarz rearrangement of $f$ (see Definition \ref{rear}). 
An interesting question is to consider the class of domains $\Omega$ with prescribed volume, a given radial function $f^\sharp$ on $\Om^\sharp$ and optimize a functional of the type
\vspace{2mm}
\begin{equation}
    \label{max:interpretation}
   \mathcal{F}(\Omega):=\dfrac{\norma{u}_{X(\Om)}}{\norma{f}_{Y(\Om)}},
\end{equation}
where $\norma{\cdot}_{X(\Om)}$ and $\norma{\cdot}_{Y(\Om)}$ are suitable norms that have to be specified, $u$ solves the Poisson problem \eqref{main_problem} with a given datum $f$ such that its Schwarz rearrangement is $f^\sharp$.

In the case of homogeneous Dirichlet boundary conditions, this question is addressed in the classical work of Talenti  \cite{talenti76}. Given an open and bounded set $\Omega\subset\mathbb{R}^n$, the author considers the problems
\begin{equation*}
    \begin{cases}
        -\Delta w =f &\text{in}\;\Omega\\
        w=0 &\text{on}\;\partial\Omega,
    \end{cases} \quad \quad  \begin{cases}
        -\Delta z =f^{\sharp} &\text{in}\;\Omega^{\sharp}\\
        z=0 &\text{on}\;\partial\Omega^{\sharp},
    \end{cases}
\end{equation*}
with $f\in L^{\frac{2n}{n+2}}(\Omega)$ if $n>2$, or  $f\in L^p(\Omega)$, with  $p>1$, if $n=2$. He proves the point-wise inequality
\begin{equation}\label{tal:diri}
    w^\sharp(x)\le z(x) \quad  \forall x \in \Omega^\sharp.  
\end{equation}
From  inequality \eqref{tal:diri}, it follows that, once  the measure of $\Omega$ and the rearrangement of $f$ are prescribed, the ratio giving $\norma{w}_{X(\Omega)}/\norma{f}_{Y(\Omega)}$ (i.e. the functional in \eqref{max:interpretation})
is maximized by the function $z$ that solves the Poisson problem on the ball. Notice that, in this case, both norms can be chosen to be any rearrangement-invariant norm. 

There are several generalization of this classical result, see, for instance, \cite{betta_mercaldo,T2} for nonlinear operators in divergence form, \cite{AFLT} for anisotropic elliptic operators, \cite{ALT} for the parabolic case,   \cite{AB, T3} for higher-order operators,  and \cite{diaz},  where the Steiner symmetrization is used in place of the Schwarz one. 

 All the above mentioned results refer to Dirichlet boundary conditions, that usually well behave under symmetrization. Actually,  for a long time, it was thought that comparison results could not be established through spherical rearrangement arguments in the case of Robin boundary conditions. However, in \cite{ANT} the authors prove a comparison involving the Lorentz norms of $u$ and $v$, under the assumptions that $f$ is a non-negative function in $L^2(\Omega)$ and that $\beta$ is a positive parameter, thereby obtaining the inequalities

\begin{equation}
	    \label{diseq_f_generica}
	    \norma{u}_{L^{k,1}(\Omega)} \, \leq \norma{v}_{L^{k,1}(\Omega^\sharp)} \, \; \forall \, 0 < k \leq \frac{n}{2n-2},
	\end{equation}
    and
	\begin{equation}
	    \label{fgen2}
	    \norma{u}_{L^{2k,2}(\Omega)} \, \leq \norma{v}_{L^{2k,2}(\Omega^\sharp)} \, \; \forall \, 0 < k \leq \frac{n}{3n-4}.
	\end{equation}
Moreover, in the case $f\equiv 1$, they also prove
\begin{equation*}
    \norma{u}_{L^p(\Omega)}\le \norma{v}_{L^p(\Omega^\sharp)}, \quad p=1,2,
\end{equation*}
 and, if $n=2$, the pointwise comparison holds:
\begin{equation}\label{trombetti}
    u^\sharp(x) \le v(x), \quad \text{ for all } x\in\Omega^\sharp.
\end{equation}
This result makes it possible to solve the optimization problem \eqref{max:interpretation} when the norm of \SSS$u$ \BBB coincides with one of the norms appearing in inequalities \eqref{diseq_f_generica} and \eqref{fgen2}. 
However, differently from the Dirichlet case, in this setting one cannot choose an arbitrary rearrangement-invariant norm, except in the case 
 $f\equiv1$ and $n=2$ , where \eqref{trombetti} holds.

Also for Robin boundary conditions,  different generalizations of these comparison results have been studied: see \cite{AGM} for the $p$-Laplacian, \cite{yabo, San2} for the anisotropic case,  \cite{ACNT,amato2022isoperimetric} for mixed boundary conditions, \cite{nunzia2022sharp} for the Hermite operator. 


 Once obtained the comparison results in terms of the inequalities above,  the next natural step 
is to characterize the equality case,  namely the rigidity of the inequalities.  In the case of Dirichlet boundary conditions, the rigidity is proved in \cite{lions_remark}, while, in the case of Robin boundary conditions, the equality case in 
\eqref{fgen2} is characterized in \cite{mp2}, where the authors proved that equality holds if and only if the set $\Omega$ is a ball and the functions $u$ and $f$ are radially symmetric and decreasing.

 Nevertheless, the characterization of the equality case in \eqref{diseq_f_generica} is left by the authors in \cite{mp2} as an open problem. This is the starting point of our work. 
Indeed, the first Theorem that we prove is the following.

\begin{teorema}\label{teo:rigidity:k1}
Let $\Omega\subseteq\R^n$ be an open, bounded and Lipschitz set, let $f\in L^2(\Omega)$ be a nonnegative function. Let $u$ be the solution to \eqref{main_problem} and let $v$ be the solution to \eqref{sym_problem}, if 
   \begin{equation}
       \label{ip:rigidity:k1}
        \norma{v}_{L^{k,1}(\Omega^\sharp)}=\norma{u}_{L^{k,1}(\Omega)},
   \end{equation}
   for some $\displaystyle{0<k\le \frac{n}{2n-2}}$,
    then,  there exists $x_0\in \R^n$ such that
\begin{equation*}
    \Omega=\Omega^\sharp +x_0, \qquad u(\cdot+x_0)=u^\sharp(\cdot), \qquad f(\cdot+ x_0)=f^\sharp(\cdot).
\end{equation*}
    \end{teorema}

Once a comparison and a rigidity result hold, we can inquire whether the comparison can be improved, asking:
\begin{quote} \it
     if equality \emph{almost} holds in \eqref{diseq_f_generica} and \eqref{fgen2}, is it true that the set $\Omega$ is \emph{almost} a ball, and the function $u$ and $f$ are \emph{almost} radially symmetric and decreasing?
\end{quote}

In the case of Dirichlet boundary condition, a quantitative result of this type was proved in \cite{ABMP}, where the authors show that whenever the comparison \eqref{tal:diri} is in force, then there exist some positive constants $ \theta_1=\theta_1(n), \theta_2=\theta_2(n)$ and $ K_1:= K_1(n, \abs{\Omega}, f^\sharp),\, K_2:=  K_2(n, \abs{\Omega}, f^\sharp),\, K_3:=  K_3(n, \abs{\Omega}, f^\sharp),$ such that
\begin{equation}\label{esti_tot}
     ||z-w^{\sharp}||_{L^\infty(\Omega)}\ge K_1 \alpha^3(\Omega)+ K_2 \inf_{x_0\in\R^n}\norma{w-w^\sharp(\cdot+x_0)}_{L^1(\R^n)}^{ \theta_1}\hspace{-1mm} +K_3\inf_{x_0\in\R^n} \norma{f-f^\sharp(\cdot +x_0)}_{L^1(\R^n)}^{ \theta_2}.
\end{equation}
Moreover, the dependence of $ K_1,\,  K_2$ and $  K_3$ on $\abs{\Omega}$ and $f^\sharp$ is explicit.
In this result, the distance between the functions $u$ and $f$ from their symmetrized is measured in terms of the $L^1$ distance, while the distance of $\Omega$ from being a  ball is measured via the \emph{Fraenkel Asymmetry}, that is 

 \begin{equation}\label{asimm}
	\alpha(\Omega):=\min_{x \in \R^{n}}\bigg \{  \dfrac{|\Omega\Delta B_r(x)|}{|B_r(x)|} \;,\; |B_r(x)|=|\Omega|\bigg \},
\end{equation}
where the symbol $\Delta$ stands for the symmetric difference. 

The authors in \cite{ABMP} also derive an analogous estimate for the difference of  
$L^p$-norm of the function, extending to any right-hand side $f$ the result proved for 
$f=1$ in \cite{kim}, which reads as follows:\begin{equation}\label{lpdiri}
    \norma{z}_p^p-\lVert w^\sharp\rVert_p^p \geq K_4\alpha^{2+p}(\Omega) \qquad \forall p >1.
     \end{equation}

It is also worth noting that in \cite{paolo} the authors refine the estimate in \eqref{esti_tot} in the special case where $\Omega$ is a ball (so that $\alpha(\Omega)=0$) and the datum $f$ is a characteristic function, obtaining the optimal exponent in
$$ \norma{z}_p^p-\lVert w^\sharp\rVert_p^p \geq K\inf_{x_0\in\R^n} \norma{f-f^\sharp(\cdot +x_0)}_{L^1(\R^n)}^{ 2}\qquad \forall p >1.$$
For further generalizations, see also \cite{amato_barbato_2024, romeo, MS}.
 
The main objective of this paper is to study the stability of \eqref{diseq_f_generica} and \eqref{fgen2} in the same spirit as the results in \cite{ABMP}. First of all, we improve inequality \eqref{diseq_f_generica} by adding a reminder term that measures the distance of the set $\Omega$ from a ball. This is our second main result

\begin{teorema}
\label{teo:quant_k1}
Let $\Omega$ be an open, bounded, and Lipschitz set, let $f\in L^2(\Omega)$ be a nonnegative function. Let $u$ be the solution to \eqref{main_problem} and let $v$ be the solution to \eqref{sym_problem}, then there exists a positive constant $C_1=C_1(\abs{\Omega},\norma{f}_{1},\beta,n)$ such that, for $0 <k\le\frac{n}{2n-2}$,

    \begin{equation}\label{quant_lk1}
        \norma{v}_{L^{k,1}(\Omega^\sharp)}-\norma{u}_{L^{k,1}(\Omega)}\ge C_1 \alpha^2(\Omega),
    \end{equation}
    where $$C_1=\frac{\abs{\Omega}^{\frac{1}{k}+\frac{1}{n}-1}\norma{f}_1}{\beta n\omega_n^\frac{1}{n}}\min\left\{\frac{1}{2^{\frac{1}{k}+5}\gamma_n}, \frac{\beta\abs{\Omega}^\frac{1}{n}}{2^{\frac{1}{k}+3+\frac{2}{n}}n\omega_n^\frac{1}{n}}\right\},$$

     $\omega_n$ is the measure of the unit ball of $\R^n$ and $\gamma_n$ is the constant appearing in the quantitative isoperimetric inequality (see Theorem \ref{quant_isop_prop} ).
\end{teorema}

Then, we perform the same analysis for the inequality \eqref{fgen2}, obtaining the following.
\begin{teorema}\label{teo:quant:k2}
Let $\Omega\subseteq\R^n$ be an open, bounded, and Lipschitz set, let $f\in L^2(\Omega)$ be a nonnegative function.
    Let $u$ be the solution to \eqref{main_problem} and let $v$ be the solution to \eqref{sym_problem}, then there exists a positive constant $C_2=C_2(\abs{\Omega},\norma{f}_1,\beta,n)$ such that for $0<k\le\frac{n}{3n-4}$,

    \begin{equation}\label{quant_l2k}
        \norma{v}_{L^{2k,2}(\Omega^\sharp)}^2-\norma{u}_{L^{2k,2}(\Omega)}^2\ge C_2 \alpha^2(\Omega),
    \end{equation}
    where $$C_2=\left(\frac{\abs{\Omega}^{\frac{1}{n}-1}\norma{f}_1}{\beta n\omega_n^\frac{1}{n}}\right)^2\abs{\Omega}^\frac{1}{k}\min\left\{\frac{1}{2^{\frac{1}{k}+5}\gamma_n}, \frac{\beta\abs{\Omega}^\frac{1}{n}}{2^{\frac{1}{k}+5+\frac{2}{n}}n\omega_n^\frac{1}{n}}\right\},$$
     $\omega_n$ is the measure of the unit ball of $\R^n$ and $\gamma_n$ is the constant appearing in the quantitative isoperimetric inequality (see Theorem \ref{quant_isop_prop} ).
\end{teorema}

 Using the above Theorems \ref{teo:quant_k1} and \ref{teo:quant:k2} we are able to deduce an alternative, immediate, proof of the quantitative versions of two classical inequalities within the class Lipschitz domains: the Saint-Venant inequality for the Robin torsional rigidity (Corollary \ref{cor:sv}) and, in dimension 2, the Bossel-Daners inequality for the first Robin eigenvalue (Corollary \ref{cor:bd}). The original results are proved respectively in \cite{nahon} and \cite{BFNT_Faber_Krahn} and hold in a more general class of admissible shapes.

Before giving the last comparison result, we spend a few words about the link with the Dirichlet case, that formally corresponds to the case $\beta=+\infty$. In particular, we focus on the results involving the torsional rigidity and the first eigenvalue. Indeed, it is known that, given $\Omega$ a Lipschitz domain, for $\beta\to +\infty$ the Robin eigenvalues converge to the Dirichlet ones (see, as a reference, \cite[Chapter 4]{henrot_shape} ) and the Robin torsional rigidity to the Dirichlet one (see \cite{BO} for the torsion problem), as well. However, we observe that, we cannot recover any additional information about the Dirichlet problem in terms of quantitative estimates if we send $\beta\to +\infty$ in \eqref{quant_lk1} or \eqref{quant_l2k}, since both constants $C_1$ and $C_2$, defined in Theorem \ref{teo:quant_k1} and in Theorem \ref{teo:quant:k2} respectively, go to $0$.

Anyway, in the special case $f=1$ and $n=2$ where the point-wise comparison \eqref{trombetti} holds, we obtain a result that resembles \eqref{esti_tot}, and for which the constant does not degenerate as $\beta\to \infty$.

\begin{teorema}\label{quantitativa_puntuale}
     Let $n=2$ and $f\equiv1$  and let $u$,  $v$ be the solutions to \eqref{main_problem} and \eqref{sym_problem} respectively, and let $u^\sharp$ be the Schwarz rearrangement of $u$.
 Then, there exists a  constant $ C_3=C_3(\Omega)$, such that  
\begin{equation}\label{stima:quant:punt}
     {||v-u^{\sharp}||_{L^\infty(\Omega^\sharp)}}\geq  C_3\alpha^3(\Omega),
        \end{equation}
where  
\begin{equation*}\label{cappa1}
C_3= \abs{\Omega} \min\left\{\frac{1}{2^{7}\pi}, \frac{1}{2^8\pi \gamma_2}\right\}. \end{equation*}
\end{teorema}


In all the aforementioned results, we are able to control from above the Fraenkel asymmetry of the set $\Omega$ in terms of the corresponding comparison gap, with different powers appearing. In particular, in Theorem \ref{teo:quant_k1} and Theorem \ref{teo:quant:k2} we obtain the exponent $2$, which is lower than the one appearing in the Dirichlet case. We actually conjecture that the quadratic exponent is the optimal one, as often happens when considering quantitative inequalities of spectral type. On the other hand, in Theorem \ref{quantitativa_puntuale} we obtain the exponent $3$, as in the Dirichlet case.

Although Theorem \ref{quantitativa_puntuale} may look weaker than the previous ones due to the larger exponent, it actually allows us to recover exactly the same quantitative estimate as in the Dirichlet case, since the constant obtained there does not depend on $\beta$. This is in contrast with Theorems \ref{teo:quant_k1} and \ref{teo:quant:k2}, where the constant degenerates and tends to zero as $\beta\to \infty$. Hence, for large $\beta$, Theorem \ref{quantitativa_puntuale} is stronger.

The essential difference with the Dirichlet case \eqref{esti_tot} is that the authors in \cite{ABMP} were able to obtain a more complete result, also in terms of the distance of the functions $u$ and $f$ from their respective symmetrized. In this regard, notice that the stability proof provided in \cite{ABMP}  is strongly based on the quantitative P\'olya-Szeg\H o principle proved in \cite{polyaquantitativa}, which is not valid for general $W^{1,p}(\Omega)$ functions, not vanishing on $\partial\Omega$. 

\vspace{3mm}

The structure of the paper is the following.
In Section \ref{preliminary}, we recall some useful results about rearrangements.
In Section \ref{main_results}, we prove the main results of the paper. In Section \ref{daners}, using our Theorems,  we give an alternative proof of the Bossel-Daners and of the  Saint-Venant inequality and we collect a list of open problems.

\section{Notation and Preliminaries}
\label{preliminary}
Throughout this article, we will denote by $|\Omega|$ the Lebesgue measure of an open and bounded set of $\mathbb{R}^n$, with $n\geq 2$, and by $P(\Omega)$ the perimeter of $\Omega$. Since we are assuming that $\partial \Omega$ is Lipschitz, we have that $P(\Omega)=\mathcal{H}^{n-1}(\partial\Omega)$, where $\mathcal{H}^{n-1}$ denotes the $(n-1)-$dimensional Hausdorff measure.

We recall the classical isoperimetric inequality and we refer the reader, for example, to \cite{burago,chavel, ossy,talenti} and to the original paper by De Giorgi \cite{degiorgi}.
\begin{teorema}[Isoperimetric Inequality]
Let $E\subset \R^n$ be a set of finite perimeter. Then,
\begin{equation}
   \label{isoperimetrica}
     P(E)\ge    n \omega_n^{\frac{1}{n}} \abs{E}^{\frac{n-1}{n}},
\end{equation}
where $\omega_n$ is the measure of the unit ball in $\R^n$.
Equality occurs if and only if $E$ is a ball up to a set of measure zero.
\end{teorema}

We also recall the quantitative isoperimetric inequality, proved in \cite{fusco_maggi} (see also \cite{cicaleseleonardi, Fuglede, hall, halleco}).
\begin{teorema}
    \label{quant_isop_prop}
    There exists a constant $\gamma_n$ such that,  for any measurable set $\Omega$ of finite measure
    \begin{equation}\label{quant_isop}
        P(\Omega)\geq n\omega_n^{\frac{1}{n}}\abs{\Omega}^{\frac{n-1}{n}}\left(1+\dfrac{\alpha^2(\Omega)}{\gamma_n}\right),
    \end{equation}
    where $\alpha(\Omega)$ is defined in \eqref{asimm}.
\end{teorema}
We observe that the constant $\gamma_n$ is explicitly computed in \cite{figalli_maggi}.

To prove our main results, it is useful to estimate the asymmetry of $U_t$ in terms of the asymmetry of $\Omega$, as in the following lemma (we refer to \cite[Lemma 2.8]{brasco}). 
\begin{lemma}\label{lembrasco}
    Let $\Omega\subset\mathbb{R}^n$ be an open set with finite measure, and let $U\subset\Omega$, $\abs{U}>0$ be such that 
    \begin{equation}\label{cond}
        \dfrac{\abs{\Omega\setminus U}}{\abs{\Omega}}\leq \dfrac{1}{4}\alpha(\Omega).
        \end{equation}
                Then, we have
                \begin{equation*}
                    \alpha(U)\geq \dfrac{1}{2}\alpha(\Omega).
                \end{equation*}
\end{lemma}

For the following theorem, we refer to \cite{ambrosio2000functions}.
 \begin{teorema}[Coarea formula]
 Let $\Omega \subset \mathbb{R}^n$ be an open set with Lipschitz boundary. Let $f\in W^{1,1}_{\text{loc}}(\Omega)$ and let $u:\Omega\to\R$ be a measurable function. Then,
 \begin{equation}
   \label{coarea}
   {\displaystyle \int _{\Omega}u(x)|\nabla f(x)|dx=\int _{\mathbb {R} }dt\int_{\Omega\cap f^{-1}(t)}u(y)\, d\mathcal {H}^{n-1}}.
 \end{equation}
 \end{teorema}

 \subsection{Property of the rearrangement}
Let us recall some basic notions about rearrangements. For a general overview, see, for instance, \cite{kes}.

 \begin{definizione}\label{distribution:function}
	Let $u: \Omega \to \R$ be a measurable function, the \emph{distribution function} of $u$ is the function $\mu : [0,+\infty[\, \to [0, +\infty[$ defined as the measure of the superlevel sets of $\abs{u}$, i.e.,
	$$
	\mu(t)= \abs{\Set{x \in \Omega \, :\,  \abs{u(x)} > t}}.
	$$
\end{definizione}
\begin{definizione}\label{decreasing:rear}
	Let $u: \Omega \to \R$ be a measurable function, the \emph{decreasing rearrangement} of $u$, denoted by $u^\ast(\cdot)$, is defined as
$$u^*(s)=\inf\{t\geq 0:\mu(t)<s\}.$$
\end{definizione}
	\begin{oss}\label{inverse}
	Let us notice that the function $\mu(\cdot)$ is decreasing and right continuous, and the function $u^\ast(\cdot)$ is its generalized inverse.
    From Definitions \ref{distribution:function} and \ref{decreasing:rear}, one can prove that
	 $$u^\ast (\mu(t)) \leq t, \quad \forall t\ge 0,$$ 
  $$\mu (u^\ast(s)) \leq s \quad \forall s \ge 0.$$ 

	\end{oss}

	\begin{definizione}\label{rear}
	 The \emph{Schwartz rearrangement} of $u$ is the function $u^\sharp $ whose superlevel sets are concentric balls with the same measure as the superlevel sets of $u$. 
	\end{definizione}

		We have the following relation between $u^\sharp$ and $u^*$:
	$$u^\sharp (x)= u^*(\omega_n\abs{x}^n),$$
 and one can easily check that the functions $u$, $u^*$, and $u^\sharp$ are equi-distributed, i.e., they have the same distribution function, and there holds
\begin{equation}\label{eq_norm_sym}
\displaystyle{\norma{u}_{L^p(\Omega)}=\norma{u^*}_{L^p(0, \abs{\Omega})}=\lVert{u^\sharp}\rVert_{L^p(\Omega^\sharp)}}, \quad \text{for all } p\ge1.    
\end{equation}

We also recall the Hardy-Littlewood inequality, an important property of the decreasing rearrangement (see \cite{hardy_classico})
\begin{equation*}
 \int_{\Omega} \abs{h(x)g(x)} \, dx \le \int_{0}^{\abs{\Omega}} h^*(s) g^*(s) \, ds,
\end{equation*}
thus, choosing $h(\cdot)=\chi_{\left\lbrace\abs{u}>t\right\rbrace}$, one has
\begin{equation*}
\int_{\abs{u}>t} \abs{g(x)} \, dx \le \int_{0}^{\mu(t)} g^*(s) \, ds.
\end{equation*}
We now introduce the Lorentz spaces (see \cite{T3} for more details on this topic). 
\begin{definizione}\label{lorentz}
Let $0<p<+\infty$ and $0<q\le +\infty$. The Lorentz space $L^{p,q}(\Omega)$ is the space of those functions such that the quantity:
\begin{equation*}
       \norma{u}_{L^{p,q}}=\begin{cases}
   	 \displaystyle{ \left( \int_{0}^{\infty}  t^{q} \mu(t)^{\frac{q}{p}}\, \frac{dt}{t}\right)^{\frac{1}{q}}} & 0<q<\infty\\[2ex]
	 \displaystyle{\sup_{t>0} \, (t^p \mu(t))} & q=\infty
	\end{cases}
\end{equation*}
is finite.

\end{definizione}
	
 Let us observe that for $p=q$ the Lorentz space coincides with the $L^p$ space, as a consequence of the \emph{Cavalieri Principle}

\[
\int_\Omega \abs{u}^p =p \int_0^{+\infty} t^{p-1} \mu(t) \, dt.
\]

\subsection{The Poisson  problem with Robin boundary conditions}
Let us now recall some properties of the problems we are studying.

The solutions $u$ to problem \eqref{main_problem} and $v$ to problem \eqref{sym_problem} are both superharmonic and, as a consequence of the strong maximum principle, they achieve their minima on the boundary.
If we set 
$$u_m=\min_\Omega u, \quad v_m=\min_{\Omega^\sharp} v$$ 
the positivity of $\beta$ and the Robin boundary conditions lead to $u_m \geq 0$ and $v_m \geq 0$. Hence, $u$ and $v$ are strictly positive in the interior of $\Omega$.
Moreover, we can observe that 
\begin{equation}
		\label{minima_eq}
		u_m = \min_\Omega u \leq  \min_{\Omega^\sharp} v= v_m,
	\end{equation}
since, as a consequence of the weak formulation of \eqref{main_problem} and \eqref{sym_problem}, it holds
	\begin{equation*}
		\begin{split}
			 v_m  \text{P}(\Omega^\sharp) &= \int_{\partial \Omega^\sharp} v(x) \, d\mathcal{H}^{n-1}= \frac{1}{\beta}\int_{\Omega^\sharp} f^\sharp \, dx=\frac{1}{\beta} \int_{\Omega} f \, dx \\
			& = \int_{\partial \Omega} u(x) \, d\mathcal{H}^{n-1} \\
			&\geq u_m  \text{P}(\Omega)  \geq { u}_m \text{P}(\Omega^\sharp), 
		\end{split}
	\end{equation*}
	and it leads to
	\begin{equation}
		\label{mf}
		\mu (t) \leq \phi (t) = \abs{\Omega}, \quad \forall t \leq v_m.
	\end{equation}

    Moreover, equality holds in \eqref{minima_eq} if and only if $\Omega$ is a ball, as a consequence of the isoperimetric inequality.

\vspace{5mm}

Now, for $t\geq 0$, we introduce the following notations:
$$u_M=\sup_{\Omega}u, \quad U_t=\left\lbrace x\in \Omega : u(x)>t\right\rbrace, \quad \partial U_t^{int}=\partial U_t \cap \Omega, \quad \partial U_t^{ext}=\partial U_t \cap \partial\Omega, \quad \mu(t)=\abs{U_t}$$
and
$$v_M=\sup_\Omega v, \quad V_t=\left\lbrace x\in \Omega^\sharp : v(x)> t\right\rbrace, \quad \partial V_t^{int}=\partial V_t \cap \Omega, \quad \partial V_t^{ext}=\partial V_t \cap \partial\Omega, \quad \phi(t)=\abs{V_t}.$$

Because of the invariance of the Laplacian and of the Schwarz rearrangement of $f$ by rotation, the solution $v$ to \eqref{sym_problem} is radial and, consequently, the superlevel sets $V_t$ are balls.

\vspace{3mm}
Now, we recall some technical Lemmas, proved in \cite{ANT}, that we need in what follows. We recall the proof of Lemma \ref{key} for the reader's convenience, while we omit the proof of Lemma \ref{lemma3.3} and Lemma \ref{lem_Gronwall}. 

\begin{lemma}\label{key}
Let $u$ and $v$ be the solution to \eqref{main_problem} and \eqref{sym_problem}, then, for a.e. $t>0$, it holds
\begin{equation}\label{eq_fundamental}n^2\omega_n^{\frac{2}{n}} \phi(t)^\frac{2n-2}{n}=\left(-\phi'(t)+\dfrac{1}{\beta} \int_{\partial V_t\cap \partial\Omega^\sharp} \dfrac{1}{v} \>d\mathcal{H}^{n-1} \right)\int_0^{\phi(t)}f^*(s)ds,
\end{equation}
and for almost all $t>0$ it holds
\begin{equation}\label{ineq_fundamental_senza_alpha}
n^2\omega_n^{\frac{2}{n}} \mu(t)^\frac{2n-2}{n}\le\left(-\mu'(t)+\dfrac{1}{\beta} \int_{\partial U_t\cap \partial\Omega} \dfrac{1}{u} d\mathcal{H}^{n-1}\right)\int_0^{\mu(t)}f^*(s)ds,
\end{equation}
and
\begin{equation}\label{ineq_fundamental}
n^2\omega_n^{\frac{2}{n}} \mu(t)^\frac{2n-2}{n}\left(1+\frac{\alpha^2(U_t)}{\gamma_n}\right)\le\left(-\mu'(t)+\dfrac{1}{\beta} \int_{\partial U_t\cap \partial\Omega} \dfrac{1}{u} d\mathcal{H}^{n-1}\right)\int_0^{\mu(t)}f^*(s)ds.
\end{equation}
\end{lemma}
\vspace{2mm}

\begin{proof}

Let $t>0$ and  $h>0$, in the weak formulation of \eqref{main_problem}, we choose the following test function 
\begin{equation*}
\varphi_h(x)= \left\{
\begin{array}{ll}
0 & \mbox{if $0<u<t$}\\\\
h & \mbox{if $u> t+h$} \\\\
u-t  &\mbox{if $t<u<t+h$}.
\end{array}
\right.
\end{equation*}

Then,

\begin{equation}
\begin{aligned}
 \int_{U_t \setminus  U_{t+h}} |\nabla u|^2 \, dx + \beta h \int_{\partial U_{t+h}^{ext}} u  \,d\mathcal{H}^{n-1}+ & \beta \int_{\partial U_{t}^{ext} \setminus \partial U_{t+h}^{ext}} u (u-t) \,d\mathcal{H}^{n-1}=  \\
& \int_{U_t \setminus U_{t+h}} f (u-t) \, dx + h \int_{U_{t+h} } f \, dx,
\end{aligned}
\end{equation}
so we can divide by $h$, use the Coarea formula \eqref{coarea} and let $h$ go to $0$, obtaining for a.e. $t>0$

\begin{equation}\label{fung}
\int_{\partial U_t} g(x) \,d\mathcal{H}^{n-1} = \int_{ \partial U_t^{int}} |\nabla u| \,d\mathcal{H}^{n-1} + \beta \int_{ \partial U_t^{ext}} u\, d\mathcal{H}^{n-1} = \int_{U_t} f\, dx
\end{equation}

where
	\begin{equation*}%
	g(x)=\begin{cases}
	    	\abs{\nabla u }& \text{ if }x \in \partial U_t^{ int },\\
	\beta u & \text{ if }x \in \partial U_t^{ ext }.
	\end{cases}
	\end{equation*}
Now combining the H\"older inequality, \eqref{fung} and the Hardy-Littlewood inequality, we have  for a.e. $t \in [0, u_M)$ 
\begin{equation*}
\begin{aligned}
 P^2(u>t)  &\leq  \left(\int_{\partial U_t} g(x)d\mathcal{H}^{n-1}  \right) \left(\int_{\partial U_t}  g(x)^{-1}d\mathcal{H}^{n-1}  \right)\\
&=\left( \int_{\partial U_t} g(x)d\mathcal{H}^{n-1}  \right)\left(\int_{ \partial U_t^{ int } } |\nabla u |^{-1}d\mathcal{H}^{n-1} + \int_{ \partial U_t^{ ext }}   (\beta u )^{-1}d\mathcal{H}^{n-1} \right)\\ &\leq
\displaystyle\int_0^{\mu(t)}f^*(s)ds \left( -\mu'(t) + \dfrac{1}{\beta} \int_{ \partial U_t^{ ext }} \dfrac{1}{u} \>d\mathcal{H}^{n-1} \right).
\end{aligned}
\end{equation*}
The isoperimetric inequality gives us \eqref{ineq_fundamental_senza_alpha}, and the quantitative isoperimetric inequality gives \eqref{ineq_fundamental}.
If $v$ solves Problem \eqref{sym_problem}, all the previous inequalities hold as equalities, hence \eqref{eq_fundamental} follows.
\end{proof}

\begin{oss}
Let us observe that if we integrate \eqref{eq_fundamental}, one can recover the explicit expression of the function $v$, solution to \eqref{sym_problem}

\begin{equation}\label{v_explicit}
    v(x)=\int_{\omega_n \abs{x}^n}^{\abs{\Omega}} \dfrac{1}{n^2\omega_n^{\frac{2}{n}} t^{2-\frac{2}{n}}}\int_0^t f^*(r)\;dr\, dt+\frac{\abs{\Omega}^\frac{1}{n}}{\beta n\omega_n^{\frac{1}{n}}}\dashint_0^{\abs{\Omega}}f^*(s)\, ds,
\end{equation}

and in particular
$$v_m=\min_{\Omega^\sharp}v=\frac{\abs{\Omega}^\frac{1}{n}}{\beta n\omega_n^{\frac{1}{n}}}\dashint_0^{\abs{\Omega}}f^*(s)\, ds.$$

\end{oss}
\begin{lemma}
	\label{lemma3.3}
	For all $\tau \geq v_m $,  we have
	\begin{equation}
	\label{intmu}
	\int_0^\tau t \left(\int_{\partial U_t^{ ext } } \frac{1}{ u(x) } \, d \mathcal{H}^{n-1}\right) \, dt \leq \frac{1}{2\beta} \int_0^{\abs{\Omega}} f^\ast(s) \,ds.
	\end{equation}
	Moreover,
	\begin{equation}
	\label{intfi}
	\int_0^\tau t \left(\int_{\partial V_t \cap \partial\Omega^\sharp } \frac{1}{ v(x) } \, d \mathcal{H}^{n-1}\right) \, dt =  \frac{1}{2\beta} \int_0^{\abs{\Omega}} f^\ast(s) \,ds.
	\end{equation}
\end{lemma}

\begin{lemma}[Gronwall]\label{lem_Gronwall}
Let $\xi(\tau)$ be a continuously differentiable function,  and let $C$ be a non-negative constant such that the following differential inequality holds 
$$
	\tau \xi' (\tau ) \leq  \xi(\tau) + C \quad \forall \tau \geq \tau_0 >0.
	$$
Then, we have

\begin{equation*}\label{gron1}
\begin{aligned}
    &\xi(\tau) \leq \left(\xi(\tau_0) + C\right) \left( \frac{\tau}{\tau_0}\right) - C,  &\forall \tau \geq \tau_0,\\
&\xi'(\tau) \leq \left( \frac{\xi(\tau_0 )+ C}{\tau_0}\right) , &\forall \tau \geq \tau_0.
\end{aligned}
\end{equation*}
\end{lemma}

We conclude the preliminary section by recalling the rigidity result proved in \cite{lions_remark} for the Dirichlet boundary conditions.
\begin{teorema}
    \label{teo:alt}
    Let $f\in L^{\frac{2n}{n+2}}(\Omega)$ if $n>2$, $f\in L^p(\Omega)$, $p>1$ if $n=2$, be a nonnegative function and let $w$ and $z$ the solutions to

    \begin{equation*}
    \begin{cases}
        -\Delta w =f &\textrm{in}\;\Omega\\
        w=0 &\textrm{on}\;\partial\Omega,
    \end{cases} \quad \quad  \begin{cases}
        -\Delta z =f^{\sharp} &\textrm{in}\;\Omega^{\sharp}\\
        z=0 &\textrm{on}\;\partial\Omega^{\sharp},
    \end{cases}
\end{equation*}
respectively. If $$w^\sharp(x)=z(x) \quad \text{for almost any } x\in \Omega^\sharp,$$
then
there exists $x_0\in \R^n$ such that

\begin{equation*}
    \Omega=\Omega^\sharp +x_0, \qquad u(\cdot+x_0)=u^\sharp(\cdot), \qquad f(\cdot+ x_0)=f^\sharp(\cdot).
\end{equation*}

\end{teorema}
\section{Main results}
\label{main_results}
In this Section, we prove the main results of the paper. 
\subsection{The rigidity result}
We start by proving the rigidity result. The main idea is to prove, under our assumptions, that the solutions $u$ to problem \eqref{main_problem} and $v$  to problem \eqref{sym_problem} satisfy a suitable Poisson problem with Dirichlet boundary condition, allowing us to apply Theorem \ref{teo:alt}.

    \begin{proof}[Proof of Theorem \ref{teo:rigidity:k1}]
  Let $0 <k\le\frac{n}{2n-2}$. Let us multiply \eqref{ineq_fundamental_senza_alpha} by $t\mu(t)^{\frac{1}{k}-\frac{2n-2}{n}}$, let us integrate from $0$ to some $\tau\ge v_m$, and let us use  Lemma \ref{lemma3.3}, we obtain
\begin{equation*}
\int_0^\tau n^2\omega_n^\frac{2}{n}t\mu(t)^\frac1k \,dt \le\int_0^\tau-\mu'(t)t\mu(t)^{\delta}\left(\int_0^{\mu(t)}f^*(s)ds\right)\,dt
+\dfrac{|\Omega|^{\delta}}{2\beta^2}{\left(\int_0^{|\Omega|}f^*(s)ds\right)^2}.
\end{equation*}
Here we have  set $\displaystyle\delta=\frac{1}{k}-\frac{2n-2}{n}$.
Let us define $F(\ell)=\displaystyle{\int_0^\ell w^\delta\left(\int_0^w f^*(s)ds\right)\,dw}$, and let us integrate by parts both sides of the last inequality, after rearranging the terms, we have
\begin{equation*}
\tau F(\mu(\tau))+\tau\int_0^\tau n^2\omega_n^\frac{2}{n} \mu(t)^\frac1k dt\le \int_0^\tau F(\mu(t))dt+\int_0^\tau\int_0^t n^2\omega_n^\frac{2}{n}\mu(t)^\frac{1}{k} dr\, dt 
+M_1,
\end{equation*}
where $M_1=\displaystyle{\dfrac{|\Omega|^{\delta}}{2\beta^2}{\left(\int_0^{|\Omega|}f^*(s)ds\right)^2}}.$
Hence, if we set $$\xi(\tau)=\int_0^\tau F(\mu(t))dt+\int_0^\tau\left(\int_0^t n^2\omega_n^\frac{2}{n}\mu(r)^\frac{1}{k} dr\right)\, dt,$$  and $\tau_0=v_m$, we are in the hypothesis of Lemma \ref{lem_Gronwall}, and it holds
     \begin{align*}
            F(\mu(\tau))+\int_0^{\tau}n^2\omega_n^\frac{2}{n} \mu(t)^{\frac{1}{k}}\;dt&\leq \dfrac{1}{v_m}\left(\int_0^{v_m}F(\mu(t))\, dt+\int_0^{v_m}\int_0^t n^2\omega_n^\frac{2}{n}\mu(r)^{\frac{1}{k}}\,dr\, dt+M_1\right),
        \end{align*}
        We can argue in the same way with \eqref{eq_fundamental}, and we get
        \begin{align*}
        F(\phi(t))+\int_0^{\tau}n^2\omega_n^{\frac{2}{n}}\phi(t)^{\frac{1}{k}}\;dt=
            \dfrac{1}{v_m}\left(\int_0^{v_m}F(\phi(t))\, dt+\int_0^{v_m}\int_0^t n^2\omega_n^\frac{2}{n}\phi(r)^{\frac{1}{k}}\, dr\,dt+M_1\right).
        \end{align*}
       If we pass to the limit for $\tau \rightarrow +\infty$ in the previous inequalities, $F(\phi(\tau))\rightarrow 0$ and $F(\mu(\tau))\rightarrow 0$. Hence, we can subtract the two expressions, obtaining
       \begin{equation}\label{eq:k1:mu:phi}
           \begin{aligned}
            \int_0^{+\infty}n^2\omega_n^\frac{2}{n} \phi(t)^{\frac{1}{k}}\,dt-\int_0^{+\infty}n^2\omega_n^\frac{2}{n} \mu(t)^{\frac{1}{k}}\, dt\ge&\dfrac{1}{v_m}\int_0^{v_m}\big(F(\phi(t))-F(\mu(t))\big) \, dt\\+&\dfrac{1}{v_m}\int_0^{v_m}n^2\omega_n^\frac{2}{n}\int_0^{t}\big(\phi(r)^{\frac{1}{k}}-\mu(r)^{\frac{1}{k}}\big)\,dr\,dt .
           \end{aligned}
       \end{equation}

        Let us notice that the left-hand side is, up to a multiplicative constant,
        $$\norma{v}_{L^{k,1}(\Omega^\sharp)}-\norma{u}_{L^{k,1}(\Omega)}$$
        hence, the hypotheses \eqref{ip:rigidity:k1} ensure that
        
         $$\int_0^{v_m}\big(F(\phi(t))-F(\mu(t))\big) \, dt+\dfrac{1}{v_m}\int_0^{v_m}n^2\omega_n^\frac{2}{n}\int_0^{t}\big(\phi(r)^{\frac{1}{k}}-\mu(r)^{\frac{1}{k}}\big)\,dr\,dt=0.$$ 
         The monotonicity of $F$ and \eqref{mf}, ensure us that for almost every $t\le v_m$,  $$\phi(t)=\mu(t),$$
        and, by Definition \ref{distribution:function} of distribution function, this implies $u_m\ge v_m$, while in general, it holds $u_m\le v_m$, hence $u_m=v_m$ and $\Omega$ is a ball.\\
       Now we want to show that $u=u_m$ for $\mathcal{H}^{n-1}-x\in\partial\Omega$.  If we suppose by contradiction that $u_M>u_m,$ we have
        $$\int_{\partial \Omega}u\, d\mathcal{H}^{n-1}>\int_{\partial \Omega}u_m\, d\mathcal{H}^{n-1}=P(\Omega)u_m=P(\Omega^{\sharp}
        )v_m=\int_{\Omega^{\sharp}}v\, d\mathcal{H}^{n-1}.$$
        This means that
        $u_M=u_m$ and $u$ is constant on $\partial \Omega$ almost everywhere. 
        
        If we set $w=u-u_m$ and $z=v-v_m=v-u_m$, these two functions respectively  solve
        $$
        \begin{cases}
            -\Delta w=f & \textrm{in}\; \Omega,\\
            w=0 & \textrm{on}\;\partial \Omega,
        \end{cases}
        \qquad\qquad 
        \begin{cases}
            -\Delta z=f^{\sharp} & \textrm{in}\;\Omega^{\sharp},\\
            z=0 & \textrm{on}\; \partial \Omega^{\sharp},
        \end{cases}
        $$
        moreover,
         $$\abs{\Set{w>t}}=\mu(t+u_m), \quad \quad  \abs{\Set{z>t}}=\phi(t+v_m)=\phi(t+u_m).$$
                
        Hence, we can apply the classical Talenti comparison result \cite[Theorem 1]{talenti76}, obtaining $w^{\sharp}\leq z$ and, as a consequence,  $\mu(t)\leq \phi(t)$ for all $t\in [0,+\infty]$.
        
        On the other hand, the hypotheses \eqref{ip:rigidity:k1} implies that for almost every $t$ $$\mu(t)=\phi(t),$$ and hence $$w^\sharp=z, \quad \text{ for almost every } x\in \Omega^\sharp$$

    We can conclude with the classical result by Alvino-Lions-Trombetti, Theorem \ref{teo:alt}.
    
    \end{proof}
\subsection{Quantitative comparison in  Lorentz norm}
We proceed now with the proof of Theorem \ref{teo:quant_k1} and of Theorem \ref{teo:quant:k2}.
As the proofs of both results rely on the 
propagation of asymmetry contained in Lemma \ref{lembrasco}, we include them in the same section.\\
First of all, we prove the comparison involving the Lorentz norm $L^{k,1}(\Omega)$.
\begin{proof}[Proof of Theorem \ref{teo:quant_k1}]
    Let $0 <k\le\frac{n}{2n-2}$,
    proceeding as in the proof of Theorem \ref{teo:rigidity:k1}, we multiply \eqref{ineq_fundamental} by $t\mu(t)^{\frac{1}{k}-\frac{2n-2}{n}}$, then we integrate from $0$ to $\tau\ge v_m$, obtaining
\begin{equation}\label{inequalityk1}
\begin{aligned}
\int_0^\tau n^2\omega_n^{\frac2n} t\mu(t)^\frac1k\left(1+\frac{\alpha^2(U_t)}{\gamma_n}\right)\,dt \le& \int_0^\tau-\mu'(t)t\mu(t)^{\delta}\left(\int_0^{\mu(t)}f^*(s)ds\right)\,dt+M_1,
\end{aligned}
\end{equation}
where $M_1$ is the same constant that appears in the proof of Theorem \ref{teo:quant_k1}.

Let us consider the function  $F(\ell)$ 
as before, we can integrate by parts both sides of \eqref{inequalityk1} and rearrange the terms to obtain

\begin{equation*}
\begin{aligned}
&\tau F(\mu(\tau))+\tau\int_0^\tau n^2\omega_n^{\frac2n} \mu(t)^\frac1k\left(1+\frac{\alpha^2(U_t)}{\gamma_n}\right) dt\le\\& \int_0^\tau F(\mu(t))dt+\int_0^\tau\int_0^t n^2\omega_n^{\frac2n}\mu(r)^\frac{1}{k}\left(1+\frac{\alpha^2(U_r)}{\gamma_n}\right) dr\, dt +M_1
\end{aligned}
\end{equation*}
If we denote
$$\xi(\tau)=\int_0^\tau F(\mu(t))dt+\int_0^\tau\left(\int_0^t n^2\omega_n^{\frac2n}\mu(r)^\frac{1}{k}\left(1+\frac{\alpha^2(U_r)}{\gamma_n}\right) dr\right)\, dt,$$ 
and $\tau_0=v_m$, we are in the hypothesis of Lemma \ref{lem_Gronwall}, and we can deduce 
\begin{equation}\label{ineq_afterGu}
\begin{aligned}
&F(\mu(\tau))+\int_0^\tau n^2\omega_n^{\frac2n} \mu(t)^\frac1k\left(1+\frac{\alpha^2(U_t)}{\gamma_n}\right) dt\le\\\\
&\dfrac{1}{v_m}\Bigg(\int_0^{v_m} F(\mu(t))dt+\int_0^{v_m}\int_0^t n^2\omega_n^{\frac2n}\mu(r)^\frac{1}{k}\left(1+\frac{\alpha^2(U_r)}{\gamma_n}\right) dr\, dt 
+ M_1\,\Bigg).
\end{aligned}
\end{equation}
Repeating the same computation with \eqref{eq_fundamental}, we can deduce
\begin{equation}\label{ineq_afterGv}
\begin{aligned}
F(\phi(\tau))+\int_0^\tau n^2\omega_n^{\frac2n} \phi(t)^\frac1k dt= \dfrac{1}{v_m}\Bigg(\int_0^{v_m} F(\phi(t))dt
+\int_0^{v_m}\int_0^t n^2\omega_n^{\frac2n}\phi(r)^\frac{1}{k} dr\, dt + M_1\,\Bigg).
\end{aligned}
\end{equation}

We now subtract \eqref{ineq_afterGv} and \eqref{ineq_afterGu}, recalling that for $t\le v_m$, $\mu(t)\le \phi(t)$, 

\begin{equation}\label{phimenomu}
\begin{aligned}
&F(\phi(\tau))-F(\mu(\tau))+n^2\omega_n^{\frac2n}\int_0^\tau  (\phi(t)^\frac1k-\mu(t)^\frac{1}{k}) dt\ge \\ & \dfrac{1}{v_m}\Bigg(\int_0^{v_m} F(\phi(t))-F(\mu(t))\, dt
+n^2\omega_n^{\frac2n}\int_0^{v_m}\int_0^t (\phi(r)^\frac{1}{k}-\mu(r)^\frac{1}{k}) dr\, dt \,\Bigg)+\\
&\frac{n^2\omega_n^{\frac2n}}{\gamma_n}\Bigg(\int_0^\tau \mu(t)^\frac1k \alpha^2(U_t)\, dt- \frac{1}{v_m}\int_0^{v_m}\left(\int_0^t\mu(r)^\frac{1}{k} \alpha^2(U_r)\,dr\right)\, dt\Bigg)\ge\\
&\frac{n^2\omega_n^{\frac2n}}{\gamma_n}\Bigg(\int_0^{v_m} \mu(t)^\frac1k \alpha^2(U_t)\, dt- \frac{1}{v_m}\int_0^{v_m}\left(\int_0^t\mu(r)^\frac{1}{k} \alpha^2(U_r)\, dr\right)\, dt\Bigg).
\end{aligned}
\end{equation}

Thanks to the Fubini-Tonelli Theorem, we can rewrite the last integral in \eqref{phimenomu}
$$\frac{1}{v_m}\int_0^{v_m}\left(\int_0^t\mu(r)^\frac{1}{k} \alpha^2(U_r)\, dr\right)\, dt=\int_0^{v_m}\mu(t)^\frac{1}{k} \alpha^2(U_t)\left(\frac{v_m-t}{v_m}\right)\, dt,$$

and so, if we send $\tau\to +\infty$ in \eqref{phimenomu}, we have

\begin{equation}
\label{start:k1}
        \left[\norma{v}_{L^{k,1}(\Omega^\sharp)}-\norma{u}_{L^{k,1}(\Omega}\right]=\int_0^\tau  (\phi(t)^\frac1k-\mu(t)^\frac{1}{k}) dt\ge \frac{1}{\gamma_n}\int_0^{v_m}\frac{t}{v_m}\mu(t)^\frac{1}{k} \alpha^2(U_t)\, dt.
\end{equation}

We now distinguish the following cases.
\begin{description}
    \item[Case 1] $u_m\ge \frac{v_m}{2}$, in this case, since the superlevel set $U_t$ coincides with $\Omega$ for any $t\le u_m$, from \eqref{start:k1} we obtain 

    \begin{equation}
        \label{case1k1}
        \left[\norma{v}_{L^{k,1}(\Omega^\sharp)}-\norma{u}_{L^{k,1}(\Omega)}\right]\ge\frac{1 }{v_m\gamma_n}\int_0^{\frac{v_m}{2}}t\mu(t)^\frac{1}{k} \alpha^2(U_t)\, dt=\frac{ v_m}{8\gamma_n}\abs{\Omega}^\frac{1}{k}\alpha^2(\Omega).
    \end{equation}

    \item[Case 2] If $u_m<\frac{v_m}{2}$, we define, similarly to \cite{ABMP, brasco}, the quantity
    \begin{equation}\label{ess}
        s_\Omega=\sup\left\{t: \, \mu(t)>\abs{\Omega}\left(1-\frac{\alpha(\Omega)}{4}\right)\right\}, 
    \end{equation}
    and we observe that for $t\le s_\Omega$, the set $U_t$ verifies the hypothesis of Lemma \ref{lembrasco}. We have to distinguish two more subcases
    \end{description}
    \begin{description}
        \item[Case 2.1] $s_\Omega\ge\frac{v_m}{2}$, in this case we can proceed as in Case 1, observing that for $t\le \frac{v_m}{2}$, we are in the hypothesis of Lemma \ref{lembrasco}
        \begin{equation}
        \label{case21k1}
        \begin{aligned}
        \left[\norma{v}_{L^{k,1}(\Omega^\sharp)}-\norma{u}_{L^{k,1}(\Omega)}\right]&\ge\frac{ 1}{v_m\gamma_n}\int_0^{\frac{v_m}{2}}t\mu(t)^\frac{1}{k} \alpha^2(U_t)\, dt\\
        &\ge\frac{1}{\gamma_n}\left(\abs{\Omega}\left(1-\frac{\alpha(\Omega)}{4}\right)\right)^{\frac{1}{k}}\dfrac{v_m}{8}\frac{\alpha^2(\Omega)}{4}\\&\ge\frac{ v_m}{2^{\frac{1}{k}+5}\gamma_n}\abs{\Omega}^\frac{1}{k}\alpha^2(\Omega).
        \end{aligned}
    \end{equation}

    \item[Case 2.2] $s_\Omega<\frac{v_m}{2}$. In this case, we go back to \eqref{eq:k1:mu:phi} which implies
    \end{description}

\begin{equation}\label{caso3k1:int}
   n^2\omega_n^\frac2n \left[\norma{v}_{L^{k,1}(\Omega^\sharp)}-\norma{u}_{L^{k,1}(\Omega)}\right]\ge\frac{1}{v_m} \left(\int_0^{v_m}(F(\phi(t))-F(\mu(t)))\, dt \right),
\end{equation}
   where we have omitted the term
   $$n^2\omega_n^\frac{2}{n}\int_0^{v_m}\int_0^t \left(\phi(r)^\frac{1}{k}-\mu(t)^\frac{1}{k}\right)\, dr\,dt $$
   as it is positive. In order to conclude, we need to bound from below the right-hand side of \eqref{caso3k1:int}. First of all, we observe that $F(\phi(t))-F(\mu(t))\ge0$ for all $t\le v_m$, hence
\begin{equation*}
   n^2\omega_n^\frac2n \left[\norma{v}_{L^{k,1}(\Omega^\sharp)}-\norma{u}_{L^{k,1}(\Omega)}\right]\ge\frac{1}{v_m} \left(\int_{\frac{v_m}{2}}^{v_m}(F(\phi(t))-F(\mu(t)))\, dt \right),
\end{equation*}
   and, as we are assuming that $s_\Omega<\frac{v_m}{2}$, it holds $F(\mu(t))\le F\left(\abs{\Omega}\left(1-\frac{\alpha(\Omega)}{4}\right)\right)$, for all $t\in [v_m/2,v_m]$, hence

       \begin{equation*}
       \begin{aligned}
   n^2\omega_n^\frac2n \left[\norma{v}_{L^{k,1}(\Omega^\sharp)}-\norma{u}_{L^{k,1}(\Omega)}\right]&\ge\frac{1}{2} \left(F\left(\abs{\Omega}\right)-F\left(\abs{\Omega}\left(1-\frac{\alpha(\Omega)}{4}\right)\right) \right)\\
   &=\frac{\abs{\Omega}}{8} F'(\xi)\alpha(\Omega),
   \end{aligned}
\end{equation*}
for some $\xi\in \left[\abs{\Omega}\left(1-\frac{\alpha(\Omega)}{4}\right),\abs{\Omega}\right]$.
Moreover, we can bound $F'(\xi)$ from below

$$F'(\xi)=\xi^\delta \int_0^\xi f^*\ge \xi^{\delta+1}\frac{\norma{f}_{L^1(\Omega)}}{\abs{\Omega}}\ge \abs{\Omega}^\delta \norma{f}_{L^1(\Omega)}\left(1-\frac{\alpha(\Omega)}{4}\right)^{\delta+1}\ge \frac{\abs{\Omega}^\delta \norma{f}_{L^1(\Omega)}}{2^{\delta+1}},$$
hence
  \begin{equation*}
  \left[\norma{v}_{L^{k,1}(\Omega^\sharp)}-\norma{u}_{L^{k,1}(\Omega)}\right]\ge
  \dfrac{|\Omega|^{\delta+1}||f||_{L^1(\Omega)}}{2^{\delta+4}n^2 \omega_n^{\frac{2}{n}}}
  \alpha(\Omega)\geq \frac{\abs{\Omega}^{\frac{1}{k}+\frac{1}{n}}}{2^{\delta+5} n\omega_n^\frac1n } \beta v_m \alpha^2(\Omega).
\end{equation*}
where we have used the fact that $\dfrac{\alpha(\Omega)}{2}\geq \left(\dfrac{\alpha(\Omega)}{2}\right)^2$, since $\alpha(\Omega)<2$.

\end{proof}

We now prove the comparison involving the Lorentz norm $L^{2k,2}(\Omega)$.

\begin{proof}[Proof of Theorem \ref{teo:quant:k2}]
   Let $0 <k\le\frac{n}{3n-4}$ and let us multiply \eqref{ineq_fundamental} by $t\mu^{\delta}(t)$, with $\delta=\frac{1}{k}-\frac{2n-2}{n}$ and integrate between $0$ and $+\infty$. Hence, we have

\begin{equation}\label{talenti_mu2}
\begin{aligned}
    &n^2 \omega_{n}^{\frac{2}{n}} \int_0^{+\infty} t\mu^{\frac{1}{k}}(t)\,dt+  \frac{n^2\omega_n^{\frac{2}{n}}}{\gamma_n}\int_0^{+\infty}t\mu(t)^{\frac{1}{k}}\alpha^2(U_t)\, dt \leq \\
    \le&\int_0^{+\infty}  \left(\int_0^{\mu (t)}f^\ast (s) \, ds\right) \left( - \mu'(t) + \frac{1}{\beta}\int_{\partial U_t^{ext}} \frac{1}{u} \, d\mathcal{H}^{n-1}\right)t \mu(t)^{\delta}\, dt\\
    \le & \int_0^{+\infty} F(\mu(t)) dt+ \frac{\abs{\Omega}^\delta}{2\beta ^{2}} \left(\int_0^{\abs{\Omega}}f^\ast (s ) \, ds\right)^{2},
\end{aligned}
\end{equation}
where in the last inequality we have used $\mu(t)\le \abs{\Omega}$ and \eqref{intmu} in Lemma \ref{lemma3.3}. As far as $v$ is concerned, it holds
\begin{equation}
\label{talentifi2}
    \begin{aligned}
    & n^2 \omega_{n}^{\frac{2}{n}}\int_0^{+\infty} t\phi^{\frac{1}{k}}(t)\, dt\\ &= \int_0^{+\infty}  \left(\int_0^{\phi (t)}f^\ast (s) \, ds\right) \left( - \phi'(t) + \frac{1}{\beta}\int_{\partial V_t^{ext}} \frac{1}{u} \, d\mathcal{H}^{n-1}\right)t \phi(t)^{\delta}\, dt\\
    &= \int_0^{+\infty} F(\phi(t)) dt+ \frac{\abs{\Omega}^\delta}{2\beta ^2}\left(\int_0^{\abs{\Omega}}f^\ast (s) \, ds\right)^2.
    \end{aligned}
\end{equation}

Subtracting \eqref{talenti_mu2} and\eqref{talentifi2}, we have

\begin{equation}
    \label{intermedio}
    \norma{v}_{L^{2k,2}(\Omega)}^2-\norma{u}_{L^{2k,2}(\Omega)}^2\ge \frac{1}{n^2 \omega_n^{\frac{2}{n}}}\int_0^\infty F(\phi(t))-F(\mu(t))
dt+    
\frac{1}{\gamma_n}\int_0^{+\infty}t\mu(t)^{\frac{1}{k}}\alpha^2(U_t)\, dt. 
\end{equation}

and, we observe, that  in \cite{ANT} it is proved that
\begin{equation} \label{F-F}
    \int_0^\infty F(\phi(t))-F(\mu(t)) dt\geq 0.
\end{equation}
Now, let us distinguish the following case

\begin{description}
    \item[Case 1]: $u_{m}\geq \dfrac{v_m}{2}$. In this case, since the superlevel set $U_t$ coincides with $\Omega$ for any $t\le u_m$, it holds
    \begin{equation}
    \norma{v}_{L^{2k,2}(\Omega)}^2-\norma{u}_{L^{2k,2}(\Omega)}^2\geq \frac{1}{\gamma_n}\int_{0}^{\frac{v_m}{2}}t\mu(t)^{\frac{1}{k}}\alpha^2(U_t)\, dt= \frac{1}{\gamma_n}\abs{\Omega}^{\frac{1}{k}}\dfrac{v_m^2}{8}\alpha^2(\Omega).
    \end{equation}
    \item[Case 2:] $u_{m}<\dfrac{v_m}{2}$. In this case we define $s_{\Omega}$ as in \eqref{ess}
    and we observe that for $t\le s_\Omega$, $U_t$ satisfies the hypothesis of \ref{lembrasco}. We need to distinguish, once again, two subcases.
    \end{description}
    \begin{description}
        \item[Case 2.1] $s_{\Omega}\geq\dfrac{v_m}{2}$. Using Lemma \ref{lembrasco}, we obtain

        \begin{equation}
        \begin{aligned}
    \norma{v}_{L^{2k,2}(\Omega)}^2-\norma{u}_{L^{2k,2}(\Omega)}^2&\geq \frac{1}{\gamma_n}\int_{0}^{\frac{v_m}{2}}t\mu(t)^{\frac{1}{k}}\alpha^2(U_t)\, dt\\&\geq \frac{1}{\gamma_n}\left(\abs{\Omega}\left(1-\frac{\alpha(\Omega)}{4}\right)\right)^{\frac{1}{k}}\dfrac{v_m^2}{8}\frac{\alpha^2(\Omega)}{4}\\
    &\geq\frac{1}{\gamma_n 2^{\frac{1}{k}+5}}\abs{\Omega}^{\frac{1}{k}}{v_m^2}{\alpha^2(\Omega)}.
    \end{aligned}
    \end{equation}
    \item[Case 2.2] $s_{\Omega}<\dfrac{v_m}{2}$. We want to bound \eqref{F-F} from below.

    Let us multiply \eqref{ineq_fundamental} by $tF(\mu(t)) \mu(t)^{-2+\frac{2}{n}}$, and let us observe that for $0<k\le\frac{n}{3n-4}$ $F(\ell)\ell^{-2+\frac{2}{n}}$ is increasing.
Let us integrate between $0$ and $\tau>v_m$, we have
    \begin{equation}
\begin{aligned}
    n^2 \omega_{n}^{\frac{2}{n}} \int_0^{\tau} tF(\mu(t))\,dt\leq& \int_0^{\tau} t \mu(t)^{-2+\frac{2}{n}} F(\mu(t))\left(\int_0^{\mu (t)}f^\ast (s )\, ds\right)(-\mu'(t))\, dt+\\&+\frac{F(\abs{\Omega})\abs{\Omega}^{-2+\frac{2}{n}}}{2\beta ^{2}} \left(\int_0^{\abs{\Omega}}f^\ast (s ) \, ds\right)^{2},
\end{aligned}
\end{equation}
Let us use the following notation: 
$$M_2:=\frac{F(\abs{\Omega})\abs{\Omega}^{-2+\frac{2}{n}}}{2\beta ^{2}} \left(\int_0^{\abs{\Omega}}f^\ast (s ) \, ds\right)^{2},  $$

and 
$$
H(l)=\int_{0}^{l} w^{-\frac{2(n-1)}{n}} F(w) \biggl( \int_0^{w} f^*(s) \, ds \biggr) \, dw.
$$

We integrate by parts the first integral, obtaining 
\begin{equation}
\label{gloria}
\tau \left(n^2\omega_n^{\frac{2}{n}}\int_0^{\tau}   F (\mu(t)) \, dt + H(\mu(\tau))\right) \leq \int_{0}^{\tau} \int_0^t n^2\omega_n^{\frac{2}{n}} F(\mu(r)) \, dr dt+ \int_0^{\tau} H(\mu(t)) \, dt +M_2
\end{equation}
The previous inequality holds as an equality by replacing $\mu$ by $\phi$:
\begin{equation}
\label{paoli}
\tau \left(n^2\omega_n^{\frac{2}{n}}\int_0^{\tau}   F (\phi(t)) \, dt + H(\phi(\tau))\right) = \int_{0}^{\tau} \int_0^t n^2\omega_n^{\frac{2}{n}} F(\phi(r)) \, dr dt+ \int_0^{\tau} H(\phi(t)) \, dt +M_2
\end{equation}
We subtract \eqref{gloria} and \eqref{paoli} and apply Gronwall Lemma \ref{lem_Gronwall}  with $\tau_0=v_m$: 
\begin{equation}
\begin{aligned}
n^2\omega_n^{\frac{2}{n}}&\int_0^{\tau}   F (\phi(t)) \, dt + H(\phi(\tau))- \left(n^2\omega_n^{\frac{2}{n}}\int_0^{\tau}   F (\mu(t)) \, dt + H(\mu(\tau))\right)\ge 
    \\ &\frac{1}{v_m}\left(\int_0^{v_m}\int_{0}^t n^2\omega_n^{\frac{2}{n}}(F(\phi(r)-F(\mu(r))\right)\,dr dt+ \int_0^{v_m} H(\phi(t))-H(\mu(t))\;dt.
\end{aligned}
\end{equation}

Now, let us observe that for $t\le v_m$, $\phi(t)\ge \mu(t)$, and the function $H$ is increasing. Combining this with the limit as $\tau$ goes to $+\infty$ we obtain

\begin{equation}\label{intermedio2}
\int_0^{+\infty}   (F (\phi(t)) - F (\mu(t)) )\, dt \ge 
    \frac{1}{v_m}\left(\int_0^{v_m}\int_{0}^t (F(\phi(r)-F(\mu(r))\, dr\, dt\right).
\end{equation}

Now combining \eqref{intermedio} and \eqref{intermedio2} we have
$$n^2\omega_n^{\frac{2}{n}}\left(\norma{v}^2_{L^{2k,2}(\Omega^\sharp)}-\norma{u}^2_{L^{2k,2}(\Omega)}\right)\ge  \frac{1}{v_m}\left(\int_0^{v_m}\int_{0}^r (F(\phi(t)-F(\mu(t))\, dt\, dr\right).
 $$

The integral on the right-hand side can be bounded as follows
 \begin{equation*}
 \begin{aligned}
     \int_{\frac{v_m}{2}}^{v_m}\int_{\frac{v_m}{2}}^r (F(\phi(t))-F(\mu(t))\, dt\, dr&\ge \Bigg(F(\abs{\Omega})-F\left(\abs{\Omega}\left(1-\dfrac{\alpha}{4}\right)\right)\Bigg)\int_{\frac{v_m}{2}}^{v_m}\left(r-\frac{v_m}{2}\right) \;dr,
     \end{aligned}
 \end{equation*}
where we have used that $s_\Omega<v_m/2$, so $\mu(t)\le \abs{\Omega}\left(1-\frac{\alpha}{4}\right)$ for $t\in \left[\frac{v_m}{2}, v_m\right]$. 

Hence
\begin{equation*}
     n^2\omega_n^{\frac{2}{n}}\left(\norma{v}^2_{L^{2k,2}(\Omega^\sharp)}-\norma{u}^2_{L^{2k,2}(\Omega)}\right)\ge \frac{v_m\abs{\Omega}}{2^5} F'(\xi)\alpha(\Omega)
\end{equation*}
with $\xi\in \left(\abs{\Omega}\left(1-\frac{\alpha(\Omega)}{4}\right),\abs{\Omega}\right)$
as before, we can uniformly bound $F'(\xi)$ from below 
$$F'(\xi)\ge \frac{\abs{\Omega}^\delta \norma{f}_{L^1(\Omega)}}{2^{\delta+1}}.$$
Hence, 
 \begin{equation*}
  \left[\norma{v}^2_{L^{2k,2}(\Omega^\sharp)}-\norma{u}^2_{L^{2k,2}(\Omega)}\right]\ge
  v_m\dfrac{|\Omega|^{\delta+1}||f||_{L^1(\Omega)}}{2^{\delta+6}n^2 \omega_n^{\frac{2}{n}}}
  \alpha(\Omega)\geq  v_m^2\frac{\abs{\Omega}^{\frac{1}{k}+\frac{1}{n}}}{2^{\delta+7} n\omega_n^\frac1n } \beta \alpha^2(\Omega),
\end{equation*}
where we have used the fact that $\dfrac{\alpha(\Omega)}{2}\geq \left(\dfrac{\alpha(\Omega)}{2}\right)^2$, since $\alpha(\Omega)<2$.
     \end{description}

\end{proof}

\begin{oss} \label{rem:1}
In \cite{ANT} the authors also prove that in the case $f\equiv 1$, it holds
		\begin{equation*}
		\begin{aligned}
         &\norma{u}_{L^{k,1}(\Omega)} \leq \norma{v}_{L^{k,1}(\Omega^{\sharp})},\\
		&\norma{u}_{L^{2k,2}(\Omega)} \leq \norma{v}_{L^{2k,2}(\Omega^{\sharp})}, 
		\end{aligned}
        \quad\quad {\rm if} \:\; \displaystyle{0 <k \leq \frac{n}{n-2}},
		\end{equation*} 
  that is an improvement of \eqref{diseq_f_generica} and \eqref{fgen2}. 
		We stress that the proof of Theorems \ref{teo:quant_k1} and  \ref{teo:quant:k2} can be adapted to the case $f\equiv 1$, although now the admissible $k$ varies in a wider range.
\end{oss}

\subsection{Pointwise quantitative comparison}
We conclude the section of main results with the proof of Theorem \ref{quantitativa_puntuale}, that is the one that mostly recalls the result in \cite{ABMP}. 
\begin{proof}[Proof of Theorem \ref{quantitativa_puntuale}]
 Let us assume that $\alpha(\Omega)>0$, otherwise, the inequality \eqref{stima:quant:punt} is trivial.
We set
     \begin{equation}\label{eps}
        \varepsilon:=||v-u^{\sharp}||_{L^\infty(\Omega^\sharp)},
    \end{equation}
    and we observe that definition \eqref{eps} implies
    \begin{equation*}
      v(x)-\varepsilon  \leq u^{\sharp}(x)\leq v(x), \qquad \text{for almost any } x\in \Omega^\sharp
    \end{equation*}
  and it follows that 
\begin{equation}\label{distr}
\phi(t+\varepsilon)\leq \mu(t) \leq \phi(t).
\end{equation}
Moreover, as $\phi$ is absolutely continuous, 
we have
\begin{equation*}
    \phi(t+\varepsilon)-\phi(t)=\int_t^{t+\varepsilon}\phi'=\int_t^{t+\varepsilon}\frac{1}{\beta}\int_{\partial V_r^{ext}}\frac{1}{v}\, d\mathcal{H}^1\, dr -4\pi \varepsilon\ge -4\pi \varepsilon.
\end{equation*}
So, \eqref{distr} becomes 
\begin{equation}
    \label{muquasinu}
\phi(t)-4\pi\varepsilon\leq \mu(t) \leq \phi(t).
\end{equation}
By the definition of $s_\Omega$ \eqref{ess} and the property of decreasing rearrangement, we get
\begin{equation}
\label{muustar}
\mu(s_\Omega)=\mu\left(u^\ast\left(\abs{\Omega}\left(1-\frac{\alpha(\Omega)}{4}\right)\right)\right)   \le \abs{\Omega}\left(1-\frac{\alpha(\Omega)}{4}\right).
\end{equation}
So, combining \eqref{muquasinu}, \eqref{muustar} and  the absolutely continuity 
 of $\phi(t)$, we have
\begin{equation*}
    \begin{aligned}
    \abs{\Omega}\frac{\alpha(\Omega)}{4}&\le \abs{\Omega}-\mu(s_\Omega)\le \abs{\Omega}- \phi(s_\Omega) + 4\pi\varepsilon
    =\int_0^{s_{\Omega}}-\phi'(t)dt +4\pi\varepsilon\\&=\int_0^{s_{\Omega}} -\frac{1}{\beta}\int_{\partial V_t^{ext}}\frac{1}{v}\, d\mathcal{H}^1\, dt +4\pi(s_\Omega+\varepsilon)\leq 4\pi(s_\Omega+\varepsilon),
    \end{aligned}
\end{equation*}
that gives 
\begin{equation}
    \label{veditu}
    s_\Omega\ge \frac{\abs{\Omega}\alpha(\Omega)}{16\pi}-\varepsilon.
\end{equation}

  We observe that,  in the case $f\equiv1$ and $n=2$, \eqref{ineq_fundamental} and \eqref{eq_fundamental} read
    	\begin{equation} 
	\label{talentimu}
	4\pi\left(1+\frac{\alpha^2(U_t)}{\gamma_2}\right)  \leq  \left( - \mu'(t) + \frac{1}{\beta }\int_{\partial U_t^{ext}} \frac{1}{u} \, d\mathcal{H}^1\right)
	\end{equation}
	and
\begin{equation} 
	\label{talentiphi}
	4\pi  =  \left( - \phi'(t) + \frac{1}{\beta }\int_{\partial V_t^{ext}} \frac{1}{v} \, d\mathcal{H}^1\right).
	\end{equation}
We multiply \eqref{talentimu} and \eqref{talentiphi} by $t$ and we integrate from 0 to $\tau\ge v_m,$ 
\begin{align*}
    & 2 \pi \tau^2+\frac{4\pi}{\gamma_2}\int_0^\tau t \alpha^2(U_t)\, dt\le \int_0^\tau t(-\mu'(t))\, dt+\frac{\abs{\Omega}}{2\beta^2}\\
    & 2 \pi \tau^2=\int_0^\tau t(-\phi'(t))\, dt+\frac{\abs{\Omega}}{2\beta^2}
\end{align*}
and hence, by subtracting the two expressions and performing an integration by parts

\begin{equation}\label{64}
    \frac{4\pi}{\gamma_2}\int_0^\tau t \alpha^2(U_t)\, dt\le \tau (\phi(\tau)-\mu(\tau)) \quad \forall \tau\ge v_m,
\end{equation}
where on the right hand side we have omitted $\int_0^\tau (\mu-\phi)\, dt$ as it is negative. 

Now, we can distinguish two cases.

\begin{description}
    \item[Case 1] $s_\Omega\ge v_m$, in this case we can compute \eqref{64} for $\tau=s_\Omega$, and we can observe that from \eqref{muquasinu} it follows that  $\phi-\mu\le 4\pi \varepsilon$, obtaining
    \begin{equation}\label{65}
         \frac{1}{8\gamma_2}s_\Omega \alpha^2(\Omega)\le\varepsilon.
    \end{equation}

    If $\varepsilon\le\frac{\abs{\Omega}\alpha(\Omega)}{2^5\pi}$, we can combine \eqref{veditu} and \eqref{65} to obtain
    $$\varepsilon\ge  \frac{\abs{\Omega}}{2^8\gamma_2\pi} \alpha^3(\Omega). $$

    On the other hand, if $\varepsilon>\frac{\abs{\Omega}\alpha(\Omega)}{2^5\pi}$ it is sufficient to recall that $\alpha(\Omega)<2$ to obtain

$$\varepsilon\ge  \frac{\abs{\Omega}}{2^7\pi} \alpha^3(\Omega). $$
\item[Case 2] If $s_\Omega<v_m$, let us observe that
$\phi(s_\Omega)=\abs{\Omega}$, combining \eqref{muquasinu} and \eqref{muustar}, we get

\begin{equation}
    4\pi \varepsilon\ge \phi(s_\Omega)-\mu(s_\Omega)\ge \frac{\abs{\Omega}}{4} \alpha(\Omega),
\end{equation}
and, since $\frac{\alpha(\Omega)}{2}\geq \frac{\alpha^3(\Omega)}{8}$,we get
$$\varepsilon\ge  \frac{\abs{\Omega}}{2^6\pi} \alpha^3(\Omega). $$
\end{description}
\end{proof}

\section{Conclusions and Open Problems}\label{daners}
In this Section, we collect some applications of our results and some possible future perspectives.
\subsection{An alternative proof of the quantitative Bossel-Daners and Saint-Venant inequality}

In this section, we give an alternative proof of the quantitative version of the Saint-Venant inequality and, in the 2-dimensional case, of the Bossel-Daners inequality, using  Theorem \ref{teo:quant_k1} and Theorem \ref{teo:quant:k2}. The proof of these quantitative enhancements is based on two immediate ideas that we will recall in the following proofs. The most delicate part is to show the uniformity of the constants in front of the square of the asymmetry.
	
 We start by defining the Robin Torsional rigidity  of the set $\Omega$   
\begin{equation*}\label{torRNintro}
T_{\beta}( \Omega)= 
\max_{\substack{w\in W^{1,2}(\Omega)\\ w\not \equiv0}}
\dfrac{\left(\displaystyle\int_{\Omega}w\;dx\right)^2}{\displaystyle\int _{\Omega}|\nabla w|^2\;dx+\beta\displaystyle\int_{\partial \Omega}w^2 \;d\mathcal{H}^{n-1}}
\end{equation*}
 and we observe that the maximum is achieved by the solution $u$ to problem \eqref{main_problem} when the right-hand side $f\equiv 1$. Moreover, one can rewrite the Robin Torsional rigidity as follows
 $$T_\beta(\Omega)=\int_\Omega u\, dx. $$
  In \cite{bugia},  the authors proved that a Saint-Venant-type inequality holds  in the more general framework of $SBV$ functions, and in \cite{nahon}, it is improved in a quantitative way.  More precisely,  the authors proved the existence of a constant $\Tilde{C}$  depending only on $\beta$,  $n$ and $|\Omega|$  for which
\begin{equation}\label{nahon_quant}
 T_\beta(\Omega^\sharp)-T_\beta(\Omega)\ge \Tilde{C}\alpha^2(\Omega).   
\end{equation}

We recover the same result by means of a completely different approach.

\begin{corollario}[Quantitative Saint-Venant inequality]\label{cor:sv}
Let $\Omega$ be an open, bounded, and Lipschitz set of $\mathbb{R}^n$. Then, there exists a constant $C_4=C_4(|\Omega|, \beta, n)>0$ such that : 
\begin{equation}\label{quant_tor}
T_\beta(\Omega^\sharp)-T_\beta(\Omega)\ge C_4\alpha^2(\Omega),
\end{equation}
where $\Omega^\sharp$ is the ball with the same volume as $\Omega$.
\end{corollario}

\begin{proof}
    Inequality \eqref{quant_tor} is an immediate consequence of Theorem \ref{teo:quant_k1}, with the choice $k=1$ that is allowed, in view of the Remark \ref{rem:1}.
\end{proof}

We now recall  the variational characterization of the first Robin eigenvalue: 
\begin{equation*}
		\lambda_\beta (\Omega)=  \min_{\substack{w\in W^{1,2}(\Omega)\\w\not \equiv0}} \dfrac{\displaystyle\int _{\Omega}|\nabla w|^2\;dx+\beta\displaystyle\int_{\partial\Omega}w^2 \;d\mathcal{H}^{n-1}}{\displaystyle\int_{\Omega}w^2\;dx}.
	\end{equation*}	
The isoperimetric inequality is proved in  \cite{bossel, daners} and the quantitative result in  \cite{BFNT_Faber_Krahn}, proving that there exists a constant $\Tilde{C_1}$
\begin{equation}\label{sta_fk}
\lambda_\beta(\Omega)-\lambda_\beta(\Omega^\sharp)\ge \Tilde{C_1}\alpha^2(\Omega).
\end{equation}
Here, we give an alternative proof of \eqref{sta_fk}.

\begin{corollario}[Quantitative Bossel-Daners inequality]\label{cor:bd}
Let $\Omega$ be an open, bounded, and Lipschitz set of $\mathbb{R}^2$. Then, there exists a constant $C_5=C_5(\Omega, \beta)>0$ such that : 
\begin{equation}\label{quant_eig}
\lambda_\beta(\Omega)-\lambda_\beta(\Omega^\sharp)\ge C_5\alpha^2(\Omega),
\end{equation}
where $\Omega^\sharp$ is the ball with the same volume as $\Omega$.
\end{corollario}

\begin{proof}

To obtain \eqref{quant_eig}, we follow the same scheme as in \cite{ANT}, considering the enhanced inequality \eqref{quant_l2k} instead of the direct comparison between the two norms; subsequently, we delve into the discussion about the independence of $C_2$. Let us consider the first positive $L^2$-normalized eigenfunction $u$ of $\Omega$, namely the positive $L^2$-normalized solution of the eigenvalue problem
\begin{equation}\label{eig_problem}
    \begin{cases}
    -\Delta u=\lambda_\beta(\Omega) u \, &\text{in } \Om, \\[1ex]
    \displaystyle{\frac{\partial u}{\partial\nu}+\beta u=0} & \text{on } \partial\Omega,
    \end{cases}
\end{equation}
and consider the auxiliary problem on $\Om^\sharp$
\begin{equation}\label{aux_problem}
    \begin{cases}
    -\Delta w=\lambda_\beta(\Omega) u^\sharp \, &\text{in } \Om^\sharp, \\[1ex]
    \displaystyle{\frac{\partial w}{\partial\nu}+\beta w=0} & \text{on } \partial\Omega^\sharp.
    \end{cases}
\end{equation}
Inequality \eqref{quant_l2k} with $k=1$, that can be chosen as $n=2$, together with \eqref{eq_norm_sym}, gives
\begin{equation}\label{quant_denom}
\int_{\Om^\sharp}(u^\sharp)^2\:dx=\int_{\Om}u^2\:dx\le\int_{\Om^\sharp}w^2\:dx-C_2\alpha^2(\Omega).
\end{equation}
We now test \eqref{aux_problem} with $w$, use the Cauchy-Scwarz inequality and estimate \eqref{quant_denom}, obtaining
\begin{align*}
    \int_{\Om^\sharp}|\nabla w|^2\:dx+\beta\int_{\partial\Omega^\sharp}w^2\:d\mathcal{H}^1&=-\int_{\Om^\sharp}w\Delta w\:dx=\lambda_\beta(\Omega)\int_{\Om^\sharp}u^\sharp w\:dx\\
    &\le\lambda_\beta(\Omega)\left(\int_{\Om^\sharp}(u^\sharp)^2\:dx\right)^{1/2}\left(\int_{\Om^\sharp}w^2\:dx\right)^{1/2}\\
&\le\lambda_\beta(\Omega)\left(\int_{\Om^\sharp}w^2\:dx-C_2\alpha^2(\Omega)\right)^{1/2}\left(\int_{\Om^\sharp}w^2\:dx\right)^{1/2}\\
&\le\lambda_\beta(\Omega)\left(1-\frac{C_2}{\int_{\Om^\sharp}w^2\:dx}\alpha^2(\Omega)\right)^{1/2} \int_{\Om^\sharp}w^2\:dx\\
&\le\lambda_\beta(\Omega)\left(1-\frac{C_2}{2\int_{\Om^\sharp}w^2\:dx}\alpha^2(\Omega)\right) \int_{\Om^\sharp}w^2\:dx,\\
\end{align*}
where in the last inequality we have used the Bernoulli inequality, for which one has to suppose $\alpha(\Omega)$ sufficiently small. 

Using $w$ as a test for $\lambda_\beta(\Omega^\sharp)$ and rearranging terms we get
(quoziente di Rayleigh con $\Omega^{\sharp}$)
$$
\frac{\lambda_\beta(\Omega)-\lambda_\beta(\Omega^\sharp)}{\lambda_\beta(\Omega)}\ge\frac{C_2}{2\int_{\Om^\sharp}w^2\:dx}\alpha^2(\Omega).
$$
It remains to show that $\dfrac{C_2}{2\int_{\Om^\sharp}w^2\:dx}$ only depends on $|\Omega|,\beta$.

First of all, we recall that, in our case, $f=\lambda_\beta(\Omega) u^\sharp$, $n=2$ and $k=1$, so 
\begin{align*}
C_2&=\left(\frac{\abs{\Omega}^{\frac{1}{n}-1}\norma{f}_1}{\beta n\omega_n^\frac{1}{n}}\right)^2\abs{\Omega}^\frac{1}{k}\min\left\{\frac{1}{2^{\frac{1}{k}+5}\gamma_n}, \frac{\beta\abs{\Omega}^\frac{1}{n}}{2^{\frac{1}{k}+5+\frac{2}{n}}n\omega_n^\frac{1}{n}}\right\}
\\
&=\frac{\lambda_\beta^2(\Omega)\|u^\sharp\|_1^2}{4\beta^2\pi}\min\left\{\frac{1}{2^6\gamma_2}, \frac{\beta\abs{\Omega}^\frac{1}{2}}{2^8\sqrt{\pi}}\right\}.
\end{align*}
Now, let us notice that
\begin{align*}
w(x)&=\int_{\pi|x|^2}^{|\Omega|}\frac{1}{4\pi t}\left(\int_0^t f^*(s)\:ds\right)dt+\frac{\|f\|_1|\Omega|^{-\frac{1}{2}}}{2\beta \sqrt{\pi}}\\
&=\int_{\pi|x|^2}^{|\Omega|}\frac{1}{4\pi t}\left(\int_0^t (\lambda_\beta(\Omega) u)^*(s)\:ds\right)dt +\frac{\lambda_\beta(\Omega)\|u\|_1|\Omega|^{-\frac{1}{2}}}{2\beta \sqrt{\pi}}
\end{align*}

and so
$$\int_{\Omega^{\sharp}}w^2\:dx\le\lambda^2_\beta(\Omega)\|u^{\sharp}\|^2_1\int_{\Omega^\sharp}\left[\frac{1}{4\pi}\log\left(\frac{\abs{\Omega}}{\pi\abs{x}^2}\right)+\frac{\abs{\Omega}^{-\frac{1}{2}}}{2\beta\sqrt{\pi}}\right]^2dx.$$

Moreover, by direct computations 

$$\int_{\Omega^\sharp}\left[\frac{1}{4\pi}\log\left(\frac{\abs{\Omega}}{\pi\abs{x}^2}\right)+\frac{\abs{\Omega}^{-\frac{1}{2}}}{2\beta\sqrt{\pi}}\right]^2dx=\dfrac{1}{4\pi}\left(\dfrac{\abs{\Omega}}{2\pi}+\dfrac{1}{\pi\beta^2}+\dfrac{\sqrt{\abs{\Omega}}}{\beta \sqrt{\pi}}\right)$$
Thus, we get
\begin{equation}\label{eq:unif_cost}
\dfrac{C_2}{2\int_{\Omega^\sharp}w^2\:dx}\ge\frac{\min\left\{\frac{1}{2^6\gamma_2}, \frac{\beta\abs{\Omega}^\frac{1}{2}}{2^8\sqrt{\pi}}\right\}}{2\beta^2\left(\dfrac{\abs{\Omega}}{2\pi}+\dfrac{1}{\pi\beta^2}+\dfrac{\sqrt{\abs{\Omega}}}{\beta \sqrt{\pi}}\right)}:=C_5.
 \end{equation}
This latter is explicitly a positive constant depending only on $|\Omega|$ and $\beta$, thereby establishing the claim and concluding the proof.

\end{proof}

 Notice that, as highlighted in the the introduction, neither Corollary \ref{cor:sv} nor Corollary \ref{cor:bd} imply their quantitative Dirichlet counterpart. Indeed, letting $\beta\to+\infty$, the Robin torsional rigidity  $T_\beta(\cdot)$ and the first Robin eigenvalue $\lambda_\beta(\cdot)$ tend to the the Dirichlet torsional rigidity $T(\cdot)$ and to the first Dirichlet eigenvalue $\lambda_\beta(\cdot)$ respectively, but both constants $C_4$ and $C_5$ (and so the right-hand sides) vanish.

\subsection{Open problems}
We conclude by collecting a list of open problems arising from our results:

\begin{itemize}
    \item Is it possible to bound from below the difference of the Lorentz norm in terms of the quantity 
    $$\norma{v}_{L^{2k,2}(\Omega^\sharp)}- \norma{u}_{L^{2k,2}(\Omega)} {\ge} K_1\norma{u-u^\sharp}_{L^r}^{\theta_1}+ K_2 \norma{f- f^\sharp}_{L^s}^{\theta_2}$$
    as in the Dirichlet case? In this regard, one should find an alternative proof which does not rely on the P\'olya-Szeg\H o principle.

    \item Is it possible to obtain a quantitative estimate in terms of  the Fraenkel asymmetry with exponent $2$ instead of  $3$ in \eqref{stima:quant:punt}? We conjecture that $2$ is the optimal exponent.

\end{itemize}

\section{Acknowledgement}
The authors were partially supported by Gruppo Nazionale per l’Analisi Matematica, la Probabilità e le loro Applicazioni
(GNAMPA) of Istituto Nazionale di Alta Matematica (INdAM).  

Vincenzo Amato was partially supported by  PNRR - Missione 4 “Istruzione e Ricerca” - Componente 2 “Dalla Ricerca all'Impresa” - Investimento 1.2 “Finanziamento di progetti
presentati da giovani ricercatori” CUP:F66E25000010006.

Rosa Barbato  was supported by the Innovation Grant SCCALC, through the
Spoke 1 (Future HPC) of the National Center for HPC, Big Data, and
Quantum Computing

Simone Cito was supported by "INdAM - GNAMPA Project", codice CUP E5324001950001 and  by the Project "Stochastic Modeling of Compound Events"
Funded by the Italian Ministry of University and Research (MUR) under the PRIN 2022 PNRR (cod. P2022KZJTZ) in the framework of European Union - Next Generation EU (CUP: F53D23010060001).

Gloria Paoli was supported by "INdAM - GNAMPA Project", codice CUP E5324001950001
    and  by the Project PRIN 2022 PNRR:  "A sustainable and trusted Transfer Learning platform for Edge Intelligence (STRUDEL)", CUP E53D23016390001, in the framework of European Union - Next Generation EU program
\addcontentsline{toc}{chapter}{Bibliografia}
\bibliographystyle{plain}
\bibliography{bibiografia}

\Addresses

\end{document}